\documentclass{elsarticle}
\usepackage{amssymb, verbatim, amsthm, amsfonts, amsmath}

\usepackage{pdfpages}
\usepackage[utf8]{inputenc}

\newtheorem{thm}{Theorem}
\newtheorem{cor}{Corollary}

\newcommand{\fq}{\mathbb{F}_q}
\newcommand{\F}{\mathbb{F}}
\newcommand{\rmv}[1]{}

\begin{document}


\begin{frontmatter}
\title{A note on inverses of cyclotomic mapping permutation polynomials over finite fields\tnoteref{t1}}
\tnotetext[t1]{Qiang Wang's research is partially supported by NSERC of Canada. 
}

\author[Wang]{Qiang Wang}
\ead{wang@math.carleton.ca}
\address[Wang]{School of Mathematics and Statistics,
Carleton University,
Ottawa, ON K1S 5B6,
Canada}
\cortext[cor1]{Corresponding author.}


\rmv{
\title[inverses of cyclotomic mapping permutation polynomials]{A note on inverses of cyclotomic mapping permutation polynomials over finite fields}

\author[Wang]{Qiang Wang}
\thanks{Qiang Wang's research is partially supported by NSERC of Canada.}
\address{School of Mathematics and Statistics,
Carleton University, Ottawa, Ontario,\\ $K1S$ $5B6$, CANADA}

\email{wang@math.carleton.ca}


}

\begin{abstract}
In this note, we give a shorter proof of the result of Zheng, Yu, and Pei on the explicit formula of inverses of generalized cyclotomic permutation polynomials over finite fields. Moreover, we characterize all these cyclotomic permutation polynomials that are involutions. Our results provide a fast algorithm (only modular operations are involved) to generate many classes of generalized cyclotomic permutation polynomials, their inverses, and involutions. 
\end{abstract}




\end{frontmatter}


 Let $q = p^m$ be the power of a prime number $p$, $\F_q$ be a finite field with $q$ elements, and  $\F_q[x]$ be the ring of polynomials over $\F_q$. 
 We denote the composition of two polynomials $P(x), Q(x) \in \F[x]$ by $(P \circ Q) (x) := P(Q(x))$. We call $P(x) \in \F_q[x]$ a {\em permutation polynomial} (PP) of $\F_q$ 
 if it induces a permutation of $\F_q$. Note that since $x^q = x$ for all $x \in \F_q$, one only needs to consider polynomials of degree less than $q$.
 It is clear that permutation polynomials form a group under composition and reduction modulo $x^q - x$ that is isomorphic to the symmetric group on $q$ letters. Thus for any permutation polynomial $P(x) \in \F_q[x]$,
 there exists a unique $P^{-1}(x) \in \F_q[x]$ such that $P^{-1}(P(x)) \equiv P(P^{-1}(x)) \equiv x \pmod{x^q - x}$. Here $P^{-1}(x)$ is defined as the {\em compositional inverse} of $P(x)$,
 although we may simply call it sometimes the {\em inverse} of $P(x)$ on $\F_q$. 

In \cite{Mullen:91}, Mullen posed the problem of computing the coefficients of
the inverse polynomial of a permutation polynomial efficiently
(Problem 10). In fact, there are very few known permutation polynomials whose
 explicit compositional inverses have been obtained, and the resulting expressions are usually of a complicated nature except for the classes of the 
 permutation linear polynomials, 
 monomials, Dickson
 polynomials. 
Among them,  see \cite{WuLiu:13, Wu:14} for the inverses of linearized PPs and see \cite{CoulterHenderson:02, WuLiu:13}
for inverses of some classes of bilinear PPs.  A generalization of these results can be found in \cite{TuxanidyWang:14}. We note that each polynomial can be written uniquely in the form 
$x^rf(x^{(q-1)/\ell})$ where $\ell \mid (q-1)$ is so-called the index of the polynomial \cite{AkbaryGhiocaWang:09}.  Although the explicit 
  characterization of the inverses of PPs of  the  form $x^rf(x^{(q-1)/\ell})$ over $\F_q$ can be found in \cite{Muratovic:07} and \cite{Wang:09},  computing these inverses of PPs in general is still not efficient if the index $\ell$ is large.

Let $\gamma$ be a fixed primitive element of $\mathbb{F}_q$ throughout the paper, $\ell \mid q-1$, and the set
of all nonzero $\ell$-th powers be $C_0$. Then $C_0$ is a subgroup of $\mathbb{F}_q^*$ of
index $\ell$. The elements of the factor group $\mathbb{F}_q^*/C_0$
are the {\it cyclotomic cosets}
$$ C_i := \gamma^i C_0, \  \  \  i = 0, 1, \cdots, \ell-1.$$

For any $A_0, A_1, \cdots, A_{\ell-1} \in \mathbb{F}_q$ and a positive integer $r$,   the  {\it $r$-th order cyclotomic mapping $f^{r}_{A_0, A_1, \cdots,
A_{\ell-1}}$ of index $\ell$ }  from $\mathbb{F}_q$ to itself is defined by
$$
f^{r}_{A_0, A_1, \cdots,
A_{\ell-1}} (x) = 
\left\{
\begin{array}{ll} 
0, &   \mbox{if} ~ x=0; \\
A_i x^{r},   &   \mbox{if} ~x \in C_i, ~ 0\leq i \leq \ell-1. \\
\end{array}
\right.
$$

Essentially  $r$-th order cyclotomic mappings of index $\ell$ produce polynomials of the form  $x^r f(x^{(q-1)/\ell})$.  
Earlier, Niederreiter and Winterhof \cite{NiederreiterWinterhof:05} and  Wang\cite{Wang:07} have studied  these cyclotomic mapping permutations. We note that a polynomial could be written in the form of 
cyclotomic mappings with different indices $\ell$'s, however, only the least index among them is defined as the index of the polynomial \cite{AkbaryGhiocaWang:09}. 
Normally it is harder to generate permutation polynomials with large indices. However, in \cite{Wang:13}, the author extended the idea of piecewise construction to obtain more classes of PPs with large indices. More specifically, different branch functions on these cyclotomic cosets were introduced. 


For any $A_0, A_1, \cdots, A_{\ell-1} \in \mathbb{F}_q$ and monic polynomials $R_0(x), \ldots,  R_{\ell-1}(x) \in \fq[x]$  we define a 
{\it  generalized cyclotomic mapping $f^{R_0(x), R_1(x), \ldots, R_{\ell-1}(x)}_{A_0, A_1, \cdots,
A_{\ell-1}}$ of index $\ell$ }  from $\mathbb{F}_q$ to itself by
\begin{equation}\label{defn}
f^{R_0(x), R_1(x), \ldots, R_{\ell-1}(x)}_{A_0, A_1, \cdots,
A_{\ell-1}} (x) = 
\left\{
\begin{array}{ll} 
0, &   \mbox{if} ~ x=0; \\
A_i R_i(x), &  \mbox{if}~ x \in C_i, ~0\leq i \leq \ell-1. \\
\end{array}
\right.
\end{equation}
Moreover, $f^{R_0(x), R_1(x), \ldots, R_{\ell-1}(x)}_{A_0, A_1, \cdots,
A_{\ell-1}}$ is called a {\it  generalized cyclotomic mapping
 of the least index $\ell$} if the mapping can not be written as
  a cyclotomic mapping of any smaller index.
The polynomial of degree at most $q-1$ representing the cyclotomic
mapping $f^{R_0(x), R_1(x), \ldots, R_{\ell-1}(x)}_{A_0, A_1, \cdots,
A_{\ell-1}} (x) $ is called  a {\it  generalized cyclotomic mapping polynomial}. In particular, when $R_0(x) = \cdots  = R_{\ell-1} (x) = x^r$ for a positive integer $r$, we have an  {\it $r$-th order cyclotomic mapping polynomial}. 

Let $s :=\frac{q-1}{\ell}$ and $\zeta := \gamma^s$ be a primitive $\ell$-th root of unity.  It is also  shown in \cite{Wang:13}, for any given generalized cyclotomic mapping of index $\ell$ as in (\ref{defn}), 
we can find a unique polynomial $P(x)$ modulo $x^q -x$ corresponding to it. Namely,
\begin{equation}\label{correspondence}
P(x) =  \frac{1}{\ell} \sum_{i=0}^{\ell-1}  \sum_{j=0}^{\ell-1} A_i \zeta^{-ji} R_i(x) x^{js}. 
\end{equation}

Furthermore,  the permutation behaviour of $P(x)$  with simple branch functions $R_i(x) = x^{r_i}$ or $x^{r_i}f_i(x)^{\ell t_i}$ for $0\leq i \leq \ell -1$ was characterized in \cite{Wang:13}.

Other related piecewise constructions of PPs can be found in \cite{CaoHuZha:14, FernandoHou:12, YuanZheng:15}. Recently, Zheng, Yuan, and Pei  \cite{ZhengYuanPei:15} studied  inverses of some of these piecewise constructed PPs where all branch functions are also PPs of $\F_q$.  Moreover, 
Zheng, Yu, and Pei  \cite{ZhengYuPei:16} use piecewise interpolation formula to find explicit expression for inverses of generalized cyclotomic mapping permutations where all branch functions are monomials (i.e. $R_i(x) = x^{r_i}$).  
In this note we give a shorter proof of
the main theorem in \cite{ZhengYuPei:16}  using cyclotomic characteristics of these polynomials. 
We recall that $\ell \mid q-1$ and $s=\frac{q-1}{\ell}$.  Let $\gamma$ be a primitive element of $\F_q$ and $\zeta = \gamma^s$ be a primitive $\ell$-th root of unity. Let  $ind_\gamma(A_i)$ be the discrete logarithm of $A_i \in \F_q^* = \F_q\setminus \{0\}$, which is an integer between $0$ and $q-2$. 

\begin{thm}
\label{thm1}
Let $A_0, \ldots, A_{\ell-1} \in \F_q^*$. 
The polynomial  
\[
P(x) = \frac{1}{\ell} \sum_{i=0}^{\ell-1} \sum_{j=0}^{\ell-1} A_i  \zeta^{-ji}x^{r_i+js} \in \F_q[x]
\]
is  a PP of $\F_q$ if and only if
$(r_i, s) =1$ for any $i=0, 1, \ldots, \ell-1$ and   $\{  ind_\gamma (A_i) +i  r_i  \mid i =0,\ldots, \ell-1\}$ is a complete set of residues modulo $\ell$.  Moreover, the inverse of $P(x)$ 
is
\[
P^{-1}(x) = \frac{1}{\ell} \sum_{k=0}^{\ell-1}  \sum_{j=0}^{\ell-1} B_k  \zeta^{-jk}x^{r_k^\prime+js}
\]
where
$k:= \varphi(i) \equiv ind_{\gamma} (A_i) + ir_i \pmod{\ell}$,  $r_{k}^\prime \equiv r_i^{-1} \pmod{s}$, and $B_{k}= A_i^{-r_{k}^\prime}\gamma^{i(1-r_ir_{k}^\prime)}$ for each $i=0, \ldots, \ell-1$. 

\end{thm}
\begin{proof} The first part is a result in \cite{Wang:13} and we omit the proof.  Obviously $P(x) = A_i x^{r_i}$ if $x\in C_i$ and thus  $P(x)$ maps $C_i$ to $C_k$,  where  $k \equiv ind_{\gamma} (A_i) + ir_i \pmod{\ell}$ and  $A_i^s \zeta^{ir_i} = \zeta^k$. Hence $P^{-1}(x)$ is also a generalized cyclotomic mapping of the same index and it must maps $C_k$ to $C_i$ injectively.  Denote $P^{-1}(x) = B_k x^{r_k^\prime}$ if $x\in C_k$.  Therefore
\[
B_k (A_i ( \gamma^{i+j\ell})^{r_i})^{r_k^\prime} = \gamma^{i+j\ell}, ~~~j=0, 1, \ldots, s-1
\] 
Because the choices of $B_k$ and $r_k^\prime$ are independent of $j$, we must have $r_k^\prime \equiv r_i^{-1} \pmod{s}$, and $B_k= A_i^{-r_k^\prime}\gamma^{i(1-r_ir_k^\prime)}$.  
\end{proof}

We note $(r_i, s) =1$ implies that  there exists integers  $g_i, t_i$ such that $r_i g_i + s t_i =1$ for any  $i=0, 1, \ldots, \ell-1$. Therefore Theorem~\ref{thm1}  is equivalent to Theorem 3.3 in \cite{ZhengYuPei:16}, because  the coefficient of $x^{r_k+js}$ in the inverse PP is 
\begin{eqnarray*}
B_k \zeta^{-jk} & =& B_k  (A_i^s \zeta^{ir_i})^{-j} \\
&=& A_i^{-r_k^\prime-js}\gamma^{i(1-r_ir_k^\prime) -jsir_i} \\
& =& A_i^{-(r_k^\prime+js)} \zeta^{-i( \frac{r_ig_i-1}{s} + jr_i)} \\
& = & A_i^{-(r_k^\prime+js)} \zeta^{-i( t_k + jr_i)}  .
\end{eqnarray*}

Our Theorem~\ref{thm1} gives more information on  these inverse PPs because the image sets are specified through these cyclotomic mappings. For example, this can help us to characterize when these PPs are involutions,  which are PPs such that their compositional inverse are themselves. Involutions have been used frequently in block cipher designs (as S-Boxes).  One immediate practical advantage of involutions is that implementation of the inverse
does not require additional resources, which is particularly useful for its implementation (as part of a block cipher) in devices with limited resources.  More details can be found in  \cite{CharpinMesnagerSumanta:15} and \cite{CharpinMesnagerSumanta:16}.  The notaton  $m = n \bmod \ell
$ means $m$ is a positive integer such that $0\leq m < \ell$ and $m\equiv n \pmod{\ell}$.  

\begin{thm}\label{thm2}
Let $A_0, \ldots, A_{\ell-1} \in \F_q^*$. 
The  polynomial  
\[
P(x) = \frac{1}{\ell} \sum_{i=0}^{\ell-1}  \sum_{j=0}^{\ell-1} A_i  \zeta^{-ji}x^{r_i+js} \in \F_q[x]
\]
is  an involution  of $\F_q$  if and only if 
(i)  there exists integers  $g_i, t_i$ such that $r_i g_i + s t_i =1$ for any  $i=0, 1, \ldots, \ell-1$; 
(ii)   $\{ \varphi(i) = ind_\gamma (A_i) + i r_i  \bmod \ell \mid i =0,\ldots, \ell-1\} = \mathbb{Z}_{\ell}$,  a complete set of residues modulo $\ell$;
(iii) $r_{\varphi(i)} \equiv r_i^{-1} \equiv g_i \pmod{s} $ for any $i=0, 1, \ldots, \ell-1$;
(iv) $A_{\varphi(i)} = A_i^{-g_i} \zeta^{it_i}$  for any $i=0, 1, \ldots, \ell-1$. 
\end{thm}
\begin{proof}
By Theorem~\ref{thm1},  $P(x)$ is a PP if and only if (i) and (ii) are both satisfied. Moreover, $P(x) = P(x)^{-1}$ if and only if  $r_{\varphi(i)} \equiv  r_{\varphi(i)}^\prime \equiv r_i^{-1} \equiv g_i \pmod{s}$ and $ A_{\varphi(i)} = B_{\varphi(i)} = A_i^{-r_{\varphi(i)}^\prime} \gamma^{ist_i} = A_i^{-g_i} \zeta^{it_i} $. 
\end{proof}

The following generalization of Corollary 4.3 in \cite{ZhengYuPei:16} corresponds to the  case $\ell =2$. 
\begin{cor}
Let $q$ be an odd prime power and $s = (q-1)/2$. Let $r_0, r_1 \in \mathbb{Z}$. Then
\[ P(x) = \frac{1}{2} x^{r_0}(1+x^s) + \frac{1}{2} x^{r_1} (1-x^s) \]
is an involution of $\F_q$ if and only if $r_i = r_i^\prime = r_i^{-1} \pmod{s}$ with $i=0, 1$ and $r_1$  is odd.
\end{cor}
\begin{proof}
 Indeed,  $P(x) = x^{r_0}$ if $x\in C_0$ and $P(x)=x^{r_1}$ if $x\in C_1$. 
Here $\ell =2$ and both $A_0 = A_1 = 1$
imply that  $0 = 0+0\cdot r_0 \mod 2$ and $1 = 0+r_1 \pmod 2$. That is, $C_0$ is mapped to $C_0$ and $C_1$ is mapped to $C_1$. Hence $r_0^\prime = r_0^{-1} \pmod{s}$ and $r_1^\prime = r_1^{-1} \pmod{s}$.  Hence  $P^{-1} = P$ if and only if $r_i = r_i^\prime = r_i^{-1} \pmod{s}$ with $i=0, 1$ and $r_1$  is odd. The latter is equivalent to $s\mid r_0^2 -1$ and $2s\mid r_1^2-1$.  
\end{proof}

Theorem~\ref{thm1} and Theorem~\ref{thm2} provide an algorithmic method to generate a lot of cyclotomic permutation polynomials, inverses, and involutions of $\F_q$. We note that each one of these polynomials is uniquely determined by two sets of integers $\{r_0, \ldots, r_{\ell-1} \} $ and $\{k_0, \ldots, k_{\ell-1}\}=\{ind_{\gamma}(A_0), \ldots, ind_{\gamma}(A_{\ell-1})\} $ such that $ 1\leq r_i \leq s$ and $0\leq k_i \leq q-2$ for all $i=0, \ldots, \ell-1$, satisfying the conditions in the above theorems. 
Indeed, Theorem 1 says that the polynomial  
\[
P(x) = \frac{1}{\ell} \sum_{i=0}^{\ell-1} \sum_{j=0}^{\ell-1} \gamma^{k_i}  \zeta^{-ji}x^{r_i+js} \in \F_q[x]
\]
is  a PP of $\F_q$ if and only if
$(r_i, s) =1$ for any $i=0, 1, \ldots, \ell-1$ and   $\{   \varphi(i) = k_i +i  r_i \bmod \ell  \mid i =0,\ldots, \ell-1\} = \mathbb{Z}_{\ell}$.  Moreover, the inverse of $P(x)$ 
is
\[
P^{-1}(x) = \frac{1}{\ell} \sum_{i=0}^{\ell-1}  \sum_{j=0}^{\ell-1} \gamma^{k_{\varphi(i)}^\prime}  \zeta^{-j \varphi(i) }x^{r_{\varphi(i)}^\prime+js}
\]
where $g_i = r_i^{-1} \bmod{s}$, 
 $r_{\varphi(i)}^\prime = g_i$, and $k_{\varphi(i)}^\prime \equiv -k_ig_i + i (1-r_ig_i) \pmod{q-1}$. 
Moreover,  Theorem 2 says that a permutation polynomial $P(x)$ is an involution if and only if $r_{\varphi(i)} \equiv r_i^{-1} \pmod{s}$ and  $k_{\varphi(i)} \equiv -k_ig_i + i(1-r_ig_i)\pmod{q-1}$ for any $i=0, 1, \ldots, \ell-1$.  
Therefore the focus of our algorithm is to search for these pairs of integer tuples. It involves only modular conditions and can be easily implemented. \\



\small{

{\bf Input}: integers $q$, $n$,  and the index $\ell$. 

{\bf Output}: 

integer vectors $r$ and $k$ of length $\ell$ for PPs with index at most $\ell$; 

integer vectors $rr$ and $kk$ of length $\ell$ for the inverse of each one of above PPs;
			  
			  all involutions. 

{\bf Algorithm:}

define a finite field of size $q^n$ with a primitive element $a$

define $s= (q^n-1)/\ell$


define a vector r, rr, gg of length $\ell$ with integer values from $1$ to $s$

definte a vector k, kk of length $\ell$ with integer values from $0$ to $q^n-1$\\

\# find all PPs of index at most $\ell$

iterate $r[i]$ from $1$ to $s$ for all $i$ in $range(\ell)$ (i.e., $i$ from $0$ to $\ell-1$) 

 \hspace{0.5cm}    check if $\gcd(r[i], s)=1$ for all $ 0\leq i \leq \ell-1$

 \hspace{0.5cm}        if yes, iterate $k[i]$ from $0$ to $q^n-1$ for all $i$ in $range(\ell)$

 \hspace{1cm}  		check if the set $\{(k[i] + i*r[i])\% \ell ~ \mbox{for} ~ i~ \mbox{in}~ range(\ell) \}$ has size $\ell$ 
           	    
 \hspace{1cm}      if yes, output $r$ and $k$.  \\

\rmv{
		 \hspace{1cm}	\#output the polynomial presentation of PPs

 \hspace{1cm}			for j in range(l):

     \hspace{1cm}				for m in range(l):

    \hspace{1.5cm}            				$f = f+ a^{k[m]}*zeta^{-j*m}*x^{j*s+r[m]}$ 

	 \hspace{1cm}		$f = f/l$
}

	 \hspace{1cm}  \# compute the inverse PP with parameters $rr$ and $kk$
	 
	   \hspace{1cm}      for $i$ in $range(\ell)$
	 
	  \hspace{1.5cm}        define $gg[i] =  r[i]^{-1} ~(~mod~ s)$

 	 \hspace{1.5cm}	      define a finite map $\phi(i): i \rightarrow (k[i]+i*r[i]) \% \ell$ 

	 \hspace{1.5cm}	      $rr[{\phi(i)}] = gg[i]$ 

   	 \hspace{1.5cm}         $kk[\phi(i)] =  (i*(1-r[i]*gg[i])-k[i]*gg[i]) \% (q^n-1)$	    
   	 
   	   \hspace{1cm}       output the integer vectors $rr$ and $kk$ for the inverse PP\\

\rmv{
	 \hspace{1cm}             \#  output the inverse PP g(x)

           	 \hspace{1cm}     for jj in range(l):

	 \hspace{1cm}                    for mm in range(l):	

 \hspace{1.5cm}  $g = g+ a^{kk[0,mm]}*zeta^{-jj*mm}*x^{jj*s+rr[0,mm]}$ 
                
	 \hspace{1cm}                $g = g/l$

}

	 \hspace{1cm}  \# check whether the above PP is an involution

	 \hspace{1cm}  	       check if $r[\phi(i)] = gg[i]$ for all $i$ in $range(\ell)$
       
	 \hspace{1cm}          If yes, then 
	 
	 \hspace{1.5cm}  	   for $i$ in $range(\ell)$ 
 	
	 \hspace{2cm}  		check if $k[\phi(i)] = (- k[i]*gg[i] + i*(1-r[i]*gg[i])) \% (q^n-1)$

	 \hspace{1.5cm}    		if yes, output the involution\\
 
}

We implemented this algorithm in SAGE and generated many examples of PPs, and their inverses, and involutions.   For example, the total number of pairs of $r$ and $k$ for generalized cyclotomic PPs over $\mathbb{F}_{25}$ with index at most $\ell \leq 2$ is $4608$. Among them, the total number of involutions is $624$.  In this paper we only demonstrate the usefulness of our algorithm by providing some interesting examples over $\mathbb{F}_{25}$ such as $(r_0, r_1)  = (1, 7)$, which are included in Appendix at the end of paper.

We want to point out that our algorithm can also generate a lot of PPs with large index $\ell$. For example, let $q^n=2^6$,  $\ell = 21$, and $a^6 + a^4 + a^3 + a + 1=0$, we can easily obtain $r =  (1, 1, 1, 1, 1, 1, 1, 1, 1, 1, 1, 1, 1, 1, 1, 1, 1, 1, 1, 1, 1)$ and  $k =(0, 0, 0, 0, 0, 0, 0, 0, 0, 0,$ $ 0, 0, 0, 0, 0, 0, 0, 0, 0, 1, 62)$. Hence
$f= (a^5 + a)x^{61} + (a^5 + a^4 + a^3 + 1)x^{58} + (a^5 + a + 1)x^{55} +
(a^5 + a^3 + a^2)x^{52} + (a^3 + a^2 + a + 1)x^{49} + (a^4 + a^3)x^{46} +
(a^5 + a^4 + a^2 + a + 1)x^{43} + (a^5 + a^3 + a^2 + a)x^{40} + (a^5 + a^4
+ a^2 + a + 1)x^{37} + (a^4 + a^2 + 1)x^{34} + (a^4 + a^2 + a + 1)x^{31} +
(a^4 + a^2)x^{28} + (a^5 + a^4 + a^3 + 1)x^{25} + (a^4 + a^3)x^{22} + (a^3
+ a^2)x^{19} + (a^3 + a^2 + a)x^{16} + (a^5 + a^4 + a^2 + a)x^{13} + (a^5 +
a^4 + a^3 + a + 1)x^{10} + (a^4 + a^3 + a + 1)x^7 + (a^5 + 1)x^4 + (a^5
+ a^3 + a^2 + a)x$ is a PP of  $\mathbb{F}_{2^6}$. In this case,  $rr= (1, 1, 1, 1, 1, 1,  1, 1, 1,  1, 1, 1, 1, 1, 1, 1, 1, 1, 1, 1, 1)$ and 
$kk= (0,  0,  0,  0,  0,  0,  0,  0,  0,  0,  0,  0,  0,  0,  0,  0,  0,  0,  0,  1, 62)$. Then  $g= (a^5 + a)x^{61} + (a^5 + a^4 + a^3 + 1)x^{58} + (a^5 + a + 1)x^{55}
+ (a^5 + a^3 + a^2)x^{52} + (a^3 + a^2 + a + 1)x^{49} + (a^4 + a^3)x^{46} +
(a^5 + a^4 + a^2 + a + 1)x^{43} + (a^5 + a^3 + a^2 + a)x^{40} + (a^5 + a^4
+ a^2 + a + 1)x^{37} + (a^4 + a^2 + 1)x^{34} + (a^4 + a^2 + a + 1)x^{31} +
(a^4 + a^2)x^{28} + (a^5 + a^4 + a^3 + 1)x^{25} + (a^4 + a^3)x^{22} + (a^3
+ a^2)x^{19} + (a^3 + a^2 + a)x^{16} + (a^5 + a^4 + a^2 + a)x^{13} + (a^5 +
a^4 + a^3 + a + 1)x^{10} + (a^4 + a^3 + a + 1)x^7 + (a^5 + 1)x^4 + (a^5
+ a^3 + a^2 + a)x$ is the inverse PP of $f$. In particular,  $f$ is  an involution.

\rmv{
\begingroup
\catcode`\^^M\active
\let ^^M\par%
\input{data.txt}%
\endgroup%
}



\rmv{
\begin{table}
\begin{tabular}{l|l|l}
PP & inverse PP & Involution \\  \hline
r =  (1, 1, 1)  k =  (0, 0, 0) &  r =  (1, 1, 1)  k =  (0, 0, 0) &  Y \\ 
\end{tabular}
\caption{$\mathbb{F}_{2^6}$, $\ell = 3$}. 
\end{table}
}

It is written  in \cite{YoussefTavaresHeys:96}  that the graphs obtained by some experimental
results indicate a strong correlation between the cryptographic
properties and the number of fixed points and suggest that the
S-boxes should be chosen to contain few fixed points. This motivated Charpin, Mesnager and Sumanta to  study  involutions without fixed points in \cite{CharpinMesnagerSumanta:15}. A systematic study on involutions over the $\mathbb{F}_{2^n}$ was published later in \cite{CharpinMesnagerSumanta:16}. Here we want to point out that $0$ is always a fixed point and the number of nonzero fixed points of our  generalized cyclotomic PPs is equal to
\begin{eqnarray*}
&&\# \{ (i, j) \mid A_i(\gamma^{i+j\ell})^{r_i} = \gamma^{i+j\ell},  0 \leq i \leq \ell-1, 0\leq j \leq s-1\}
\end{eqnarray*}
We note that the condition $A_i(\gamma^{i+j\ell})^{r_i} = \gamma^{i+j\ell}$ reduces to a linear congruence 
\[
ind_\gamma(A_i) + (i+j\ell)(r_i-1) \equiv 0 \pmod{q-1}.
\]  
In order to find the number of  all fixed points, we can first determine the number of cyclotomic cosets that are fixed (i.e., all these $i$'s such that  $A_i^s \zeta^{ir_i} = \zeta^i)$. That is,  the number of $i$'s satisfying $\varphi(i)=i$ because $\varphi(i)-i \equiv ind_\gamma(A_i) + i(r_i-1) \equiv 0 \pmod{\ell}$.   For each such $i$,  we can determine the number of  $j$'s satisfying $ind_\gamma(A_i) + (i+j\ell)(r_i-1) \equiv 0 \pmod{q-1}$.  We note that $ ind_\gamma(A_i) + i(r_i-1) \equiv \ell t \pmod{q-1}$.  In this case, the above linear congruence reduces to $ t + j(r_i-1) \equiv 0 \pmod{s}$.   Therefore, the number of nonzero fixed points is equal to the sum of $1$ the number of pairs $(i, j)$ such that $\varphi(i) =i$ and  $ t + j(r_i-1) \equiv 0 \pmod{s}$.    Together with Theorem~\ref{thm2},  one  can  construct generalized cyclotomic mappings PPs that are involutions with few 
fixed points in an algorithmic way.



\vskip 1cm

\noindent {\bf Appendix:} In the following tables, $a$ is a primitive element such that $a^2 + 4a + 2 = 0$ in $\mathbb{F}_{25}$. 
The total number of generalized cyclotomic permutation polynomials of the form 
\[
 P(x) = \frac{1}{\ell} \sum_{i=0}^{\ell-1} \sum_{j=0}^{\ell-1} a^{k_i}  \zeta^{-ji}x^{r_i+js} \in \F_{25}[x] 
 \]
such that $\ell \leq 2$ is $4608$ and the total number of involutions is $624$.

In particular,  all these generalized cyclotomic permutation polynomials 
such that $r = (r_0, r_1) =  (1, 7)$ over $\mathbb{F}_{25}$ are given in tables from page 9 to page 20. From page 9 to page 14, $r$ and $k$ are pairs of integer tuples for a permutation polynomial $f(x)$. In contrast, $rr$ and $kk$ are pairs of integer tuples for the inverse polynomial of $f(x)$. The last column indicates whether $f$ is an involution or not.  The explicit polynomial expression of the inverse polynomial $f^{-1}(x)$ and its corresponding $f(x)$ are given in tables from page 15 to page 20.

Moreover, all these involutions, the corresponding pairs of integers, and the number of nonzero fixed points are given in tables from page 21 to page 33. We note that $0$ is always a fixed point for these involutions. 



\hskip -2cm
\small{

\begin{tabular}{llllll}
$r$ &  $k$ & f(x) & $rr$ & $kk$ & Involution \\  \hline
$\left(1, 7\right)$ & $\left(0, 0\right)$ & $2 x^{19} + 3 x^{13} + 3 x^{7} + 3 x$ & $\left(1,\,7\right)$ & $\left(0,\,0\right)$ & yes \\
$\left(1, 7\right)$ & $\left(0, 2\right)$ & $\left(2 a + 1\right) x^{19} + 3 x^{13} + \left(3 a + 4\right) x^{7} + 3 x$ & $\left(1,\,7\right)$ & $\left(0,\,10\right)$ & no \\
$\left(1, 7\right)$ & $\left(0, 4\right)$ & $\left(4 a + 4\right) x^{19} + 3 x^{13} + \left(a + 1\right) x^{7} + 3 x$ & $\left(1,\,7\right)$ & $\left(0,\,20\right)$ & no \\
$\left(1, 7\right)$ & $\left(0, 6\right)$ & $4 x^{19} + 3 x^{13} + x^{7} + 3 x$ & $\left(1,\,7\right)$ & $\left(0,\,6\right)$ & yes \\
$\left(1, 7\right)$ & $\left(0, 8\right)$ & $\left(4 a + 2\right) x^{19} + 3 x^{13} + \left(a + 3\right) x^{7} + 3 x$ & $\left(1,\,7\right)$ & $\left(0,\,16\right)$ & no \\
$\left(1, 7\right)$ & $\left(0, 10\right)$ & $\left(3 a + 3\right) x^{19} + 3 x^{13} + \left(2 a + 2\right) x^{7} + 3 x$ & $\left(1,\,7\right)$ & $\left(0,\,2\right)$ & no \\
$\left(1, 7\right)$ & $\left(0, 12\right)$ & $3 x^{19} + 3 x^{13} + 2 x^{7} + 3 x$ & $\left(1,\,7\right)$ & $\left(0,\,12\right)$ & yes \\
\end{tabular}

\begin{tabular}{llllll}
$r$ &  $k$ & f(x) & $rr$ & $kk$ & Involution \\ 
$\left(1, 7\right)$ & $\left(4, 0\right)$ & $2 x^{19} + \left(a + 1\right) x^{13} + 3 x^{7} + \left(a + 1\right) x$ & $\left(1,\,7\right)$ & $\left(20,\,0\right)$ & no \\
$\left(1, 7\right)$ & $\left(4, 2\right)$ & $\left(2 a + 1\right) x^{19} + \left(a + 1\right) x^{13} + \left(3 a + 4\right) x^{7} + \left(a + 1\right) x$ & $\left(1,\,7\right)$ & $\left(20,\,10\right)$ & no \\
$\left(1, 7\right)$ & $\left(4, 4\right)$ & $\left(4 a + 4\right) x^{19} + \left(a + 1\right) x^{13} + \left(a + 1\right) x^{7} + \left(a + 1\right) x$ & $\left(1,\,7\right)$ & $\left(20,\,20\right)$ & no \\
$\left(1, 7\right)$ & $\left(4, 6\right)$ & $4 x^{19} + \left(a + 1\right) x^{13} + x^{7} + \left(a + 1\right) x$ & $\left(1,\,7\right)$ & $\left(20,\,6\right)$ & no \\
$\left(1, 7\right)$ & $\left(4, 8\right)$ & $\left(4 a + 2\right) x^{19} + \left(a + 1\right) x^{13} + \left(a + 3\right) x^{7} + \left(a + 1\right) x$ & $\left(1,\,7\right)$ & $\left(20,\,16\right)$ & no \\
$\left(1, 7\right)$ & $\left(4, 10\right)$ & $\left(3 a + 3\right) x^{19} + \left(a + 1\right) x^{13} + \left(2 a + 2\right) x^{7} + \left(a + 1\right) x$ & $\left(1,\,7\right)$ & $\left(20,\,2\right)$ & no \\
$\left(1, 7\right)$ & $\left(4, 12\right)$ & $3 x^{19} + \left(a + 1\right) x^{13} + 2 x^{7} + \left(a + 1\right) x$ & $\left(1,\,7\right)$ & $\left(20,\,12\right)$ & no \\
\end{tabular}

\begin{tabular}{llllll}
$r$ &  $k$ & f(x) & $rr$ & $kk$ & Involution \\  \hline
$\left(1, 7\right)$ & $\left(8, 0\right)$ & $2 x^{19} + \left(a + 3\right) x^{13} + 3 x^{7} + \left(a + 3\right) x$ & $\left(1,\,7\right)$ & $\left(16,\,0\right)$ & no \\
$\left(1, 7\right)$ & $\left(8, 2\right)$ & $\left(2 a + 1\right) x^{19} + \left(a + 3\right) x^{13} + \left(3 a + 4\right) x^{7} + \left(a + 3\right) x$ & $\left(1,\,7\right)$ & $\left(16,\,10\right)$ & no \\
$\left(1, 7\right)$ & $\left(8, 4\right)$ & $\left(4 a + 4\right) x^{19} + \left(a + 3\right) x^{13} + \left(a + 1\right) x^{7} + \left(a + 3\right) x$ & $\left(1,\,7\right)$ & $\left(16,\,20\right)$ & no \\
$\left(1, 7\right)$ & $\left(8, 6\right)$ & $4 x^{19} + \left(a + 3\right) x^{13} + x^{7} + \left(a + 3\right) x$ & $\left(1,\,7\right)$ & $\left(16,\,6\right)$ & no \\
$\left(1, 7\right)$ & $\left(8, 8\right)$ & $\left(4 a + 2\right) x^{19} + \left(a + 3\right) x^{13} + \left(a + 3\right) x^{7} + \left(a + 3\right) x$ & $\left(1,\,7\right)$ & $\left(16,\,16\right)$ & no \\
$\left(1, 7\right)$ & $\left(8, 10\right)$ & $\left(3 a + 3\right) x^{19} + \left(a + 3\right) x^{13} + \left(2 a + 2\right) x^{7} + \left(a + 3\right) x$ & $\left(1,\,7\right)$ & $\left(16,\,2\right)$ & no \\
$\left(1, 7\right)$ & $\left(8, 12\right)$ & $3 x^{19} + \left(a + 3\right) x^{13} + 2 x^{7} + \left(a + 3\right) x$ & $\left(1,\,7\right)$ & $\left(16,\,12\right)$ & no \\
\end{tabular}

\begin{tabular}{llllll}
$r$ &  $k$ & f(x) & $rr$ & $kk$ & Involution \\  \hline
$\left(1, 7\right)$ & $\left(12, 0\right)$ & $2 x^{19} + 2 x^{13} + 3 x^{7} + 2 x$ & $\left(1,\,7\right)$ & $\left(12,\,0\right)$ & yes \\
$\left(1, 7\right)$ & $\left(12, 2\right)$ & $\left(2 a + 1\right) x^{19} + 2 x^{13} + \left(3 a + 4\right) x^{7} + 2 x$ & $\left(1,\,7\right)$ & $\left(12,\,10\right)$ & no \\
$\left(1, 7\right)$ & $\left(12, 4\right)$ & $\left(4 a + 4\right) x^{19} + 2 x^{13} + \left(a + 1\right) x^{7} + 2 x$ & $\left(1,\,7\right)$ & $\left(12,\,20\right)$ & no \\
$\left(1, 7\right)$ & $\left(12, 6\right)$ & $4 x^{19} + 2 x^{13} + x^{7} + 2 x$ & $\left(1,\,7\right)$ & $\left(12,\,6\right)$ & yes \\
$\left(1, 7\right)$ & $\left(12, 8\right)$ & $\left(4 a + 2\right) x^{19} + 2 x^{13} + \left(a + 3\right) x^{7} + 2 x$ & $\left(1,\,7\right)$ & $\left(12,\,16\right)$ & no \\
$\left(1, 7\right)$ & $\left(12, 10\right)$ & $\left(3 a + 3\right) x^{19} + 2 x^{13} + \left(2 a + 2\right) x^{7} + 2 x$ & $\left(1,\,7\right)$ & $\left(12,\,2\right)$ & no \\
$\left(1, 7\right)$ & $\left(12, 12\right)$ & $3 x^{19} + 2 x^{13} + 2 x^{7} + 2 x$ & $\left(1,\,7\right)$ & $\left(12,\,12\right)$ & yes \\
\end{tabular}

\begin{tabular}{llllll}
$r$ &  $k$ & f(x) & $rr$ & $kk$ & Involution \\  \hline
$\left(1, 7\right)$ & $\left(16, 0\right)$ & $2 x^{19} + \left(4 a + 4\right) x^{13} + 3 x^{7} + \left(4 a + 4\right) x$ & $\left(1,\,7\right)$ & $\left(8,\,0\right)$ & no \\
$\left(1, 7\right)$ & $\left(16, 2\right)$ & $\left(2 a + 1\right) x^{19} + \left(4 a + 4\right) x^{13} + \left(3 a + 4\right) x^{7} + \left(4 a + 4\right) x$ & $\left(1,\,7\right)$ & $\left(8,\,10\right)$ & no \\
$\left(1, 7\right)$ & $\left(16, 4\right)$ & $\left(4 a + 4\right) x^{19} + \left(4 a + 4\right) x^{13} + \left(a + 1\right) x^{7} + \left(4 a + 4\right) x$ & $\left(1,\,7\right)$ & $\left(8,\,20\right)$ & no \\
$\left(1, 7\right)$ & $\left(16, 6\right)$ & $4 x^{19} + \left(4 a + 4\right) x^{13} + x^{7} + \left(4 a + 4\right) x$ & $\left(1,\,7\right)$ & $\left(8,\,6\right)$ & no \\
$\left(1, 7\right)$ & $\left(16, 8\right)$ & $\left(4 a + 2\right) x^{19} + \left(4 a + 4\right) x^{13} + \left(a + 3\right) x^{7} + \left(4 a + 4\right) x$ & $\left(1,\,7\right)$ & $\left(8,\,16\right)$ & no \\
$\left(1, 7\right)$ & $\left(16, 10\right)$ & $\left(3 a + 3\right) x^{19} + \left(4 a + 4\right) x^{13} + \left(2 a + 2\right) x^{7} + \left(4 a + 4\right) x$ & $\left(1,\,7\right)$ & $\left(8,\,2\right)$ & no \\
$\left(1, 7\right)$ & $\left(16, 12\right)$ & $3 x^{19} + \left(4 a + 4\right) x^{13} + 2 x^{7} + \left(4 a + 4\right) x$ & $\left(1,\,7\right)$ & $\left(8,\,12\right)$ & no \\
\end{tabular}

\begin{tabular}{llllll}
$r$ &  $k$ & f(x) & $rr$ & $kk$ & Involution \\  \hline
$\left(1, 7\right)$ & $\left(20, 0\right)$ & $2 x^{19} + \left(4 a + 2\right) x^{13} + 3 x^{7} + \left(4 a + 2\right) x$ & $\left(1,\,7\right)$ & $\left(4,\,0\right)$ & no \\
$\left(1, 7\right)$ & $\left(20, 2\right)$ & $\left(2 a + 1\right) x^{19} + \left(4 a + 2\right) x^{13} + \left(3 a + 4\right) x^{7} + \left(4 a + 2\right) x$ & $\left(1,\,7\right)$ & $\left(4,\,10\right)$ & no \\
$\left(1, 7\right)$ & $\left(20, 4\right)$ & $\left(4 a + 4\right) x^{19} + \left(4 a + 2\right) x^{13} + \left(a + 1\right) x^{7} + \left(4 a + 2\right) x$ & $\left(1,\,7\right)$ & $\left(4,\,20\right)$ & no \\
$\left(1, 7\right)$ & $\left(20, 6\right)$ & $4 x^{19} + \left(4 a + 2\right) x^{13} + x^{7} + \left(4 a + 2\right) x$ & $\left(1,\,7\right)$ & $\left(4,\,6\right)$ & no \\
$\left(1, 7\right)$ & $\left(20, 8\right)$ & $\left(4 a + 2\right) x^{19} + \left(4 a + 2\right) x^{13} + \left(a + 3\right) x^{7} + \left(4 a + 2\right) x$ & $\left(1,\,7\right)$ & $\left(4,\,16\right)$ & no \\
$\left(1, 7\right)$ & $\left(20, 10\right)$ & $\left(3 a + 3\right) x^{19} + \left(4 a + 2\right) x^{13} + \left(2 a + 2\right) x^{7} + \left(4 a + 2\right) x$ & $\left(1,\,7\right)$ & $\left(4,\,2\right)$ & no \\
$\left(1, 7\right)$ & $\left(20, 12\right)$ & $3 x^{19} + \left(4 a + 2\right) x^{13} + 2 x^{7} + \left(4 a + 2\right) x$ & $\left(1,\,7\right)$ & $\left(4,\,12\right)$ & no \\
\end{tabular}

\begin{tabular}{ll}
$f(x)$  & $f^{-1}(x)$ \\ \hline
$2 x^{19} + 3 x^{13} + 3 x^{7} + 3 x$ & $2 x^{19} + 3 x^{13} + 3 x^{7} + 3 x$ \\
$\left(2 a + 1\right) x^{19} + 3 x^{13} + \left(3 a + 4\right) x^{7} + 3 x$ & $\left(3 a + 3\right) x^{19} + 3 x^{13} + \left(2 a + 2\right) x^{7} + 3 x$ \\
$\left(4 a + 4\right) x^{19} + 3 x^{13} + \left(a + 1\right) x^{7} + 3 x$ & $\left(a + 3\right) x^{19} + 3 x^{13} + \left(4 a + 2\right) x^{7} + 3 x$ \\
$4 x^{19} + 3 x^{13} + x^{7} + 3 x$ & $4 x^{19} + 3 x^{13} + x^{7} + 3 x$ \\
$\left(4 a + 2\right) x^{19} + 3 x^{13} + \left(a + 3\right) x^{7} + 3 x$ & $\left(a + 1\right) x^{19} + 3 x^{13} + \left(4 a + 4\right) x^{7} + 3 x$ \\
$\left(3 a + 3\right) x^{19} + 3 x^{13} + \left(2 a + 2\right) x^{7} + 3 x$ & $\left(2 a + 1\right) x^{19} + 3 x^{13} + \left(3 a + 4\right) x^{7} + 3 x$ \\
$3 x^{19} + 3 x^{13} + 2 x^{7} + 3 x$ & $3 x^{19} + 3 x^{13} + 2 x^{7} + 3 x$ \\
\end{tabular}

\begin{tabular}{ll}
$f(x)$  & $f^{-1}(x)$ \\ \hline
$2 x^{19} + \left(a + 1\right) x^{13} + 3 x^{7} + \left(a + 1\right) x$ & $2 x^{19} + \left(4 a + 2\right) x^{13} + 3 x^{7} + \left(4 a + 2\right) x$ \\
$\left(2 a + 1\right) x^{19} + \left(a + 1\right) x^{13} + \left(3 a + 4\right) x^{7} + \left(a + 1\right) x$ & $\left(3 a + 3\right) x^{19} + \left(4 a + 2\right) x^{13} + \left(2 a + 2\right) x^{7} + \left(4 a + 2\right) x$ \\
$\left(4 a + 4\right) x^{19} + \left(a + 1\right) x^{13} + \left(a + 1\right) x^{7} + \left(a + 1\right) x$ & $\left(a + 3\right) x^{19} + \left(4 a + 2\right) x^{13} + \left(4 a + 2\right) x^{7} + \left(4 a + 2\right) x$ \\
$4 x^{19} + \left(a + 1\right) x^{13} + x^{7} + \left(a + 1\right) x$ & $4 x^{19} + \left(4 a + 2\right) x^{13} + x^{7} + \left(4 a + 2\right) x$ \\
$\left(4 a + 2\right) x^{19} + \left(a + 1\right) x^{13} + \left(a + 3\right) x^{7} + \left(a + 1\right) x$ & $\left(a + 1\right) x^{19} + \left(4 a + 2\right) x^{13} + \left(4 a + 4\right) x^{7} + \left(4 a + 2\right) x$ \\
$\left(3 a + 3\right) x^{19} + \left(a + 1\right) x^{13} + \left(2 a + 2\right) x^{7} + \left(a + 1\right) x$ & $\left(2 a + 1\right) x^{19} + \left(4 a + 2\right) x^{13} + \left(3 a + 4\right) x^{7} + \left(4 a + 2\right) x$ \\
$3 x^{19} + \left(a + 1\right) x^{13} + 2 x^{7} + \left(a + 1\right) x$ & $3 x^{19} + \left(4 a + 2\right) x^{13} + 2 x^{7} + \left(4 a + 2\right) x$ \\
\end{tabular}

\begin{tabular}{ll}
$f(x)$  & $f^{-1}(x)$ \\ \hline
$2 x^{19} + \left(a + 3\right) x^{13} + 3 x^{7} + \left(a + 3\right) x$ & $2 x^{19} + \left(4 a + 4\right) x^{13} + 3 x^{7} + \left(4 a + 4\right) x$ \\
$\left(2 a + 1\right) x^{19} + \left(a + 3\right) x^{13} + \left(3 a + 4\right) x^{7} + \left(a + 3\right) x$ & $\left(3 a + 3\right) x^{19} + \left(4 a + 4\right) x^{13} + \left(2 a + 2\right) x^{7} + \left(4 a + 4\right) x$ \\
$\left(4 a + 4\right) x^{19} + \left(a + 3\right) x^{13} + \left(a + 1\right) x^{7} + \left(a + 3\right) x$ & $\left(a + 3\right) x^{19} + \left(4 a + 4\right) x^{13} + \left(4 a + 2\right) x^{7} + \left(4 a + 4\right) x$ \\
$4 x^{19} + \left(a + 3\right) x^{13} + x^{7} + \left(a + 3\right) x$ & $4 x^{19} + \left(4 a + 4\right) x^{13} + x^{7} + \left(4 a + 4\right) x$ \\
$\left(4 a + 2\right) x^{19} + \left(a + 3\right) x^{13} + \left(a + 3\right) x^{7} + \left(a + 3\right) x$ & $\left(a + 1\right) x^{19} + \left(4 a + 4\right) x^{13} + \left(4 a + 4\right) x^{7} + \left(4 a + 4\right) x$ \\
$\left(3 a + 3\right) x^{19} + \left(a + 3\right) x^{13} + \left(2 a + 2\right) x^{7} + \left(a + 3\right) x$ & $\left(2 a + 1\right) x^{19} + \left(4 a + 4\right) x^{13} + \left(3 a + 4\right) x^{7} + \left(4 a + 4\right) x$ \\
$3 x^{19} + \left(a + 3\right) x^{13} + 2 x^{7} + \left(a + 3\right) x$ & $3 x^{19} + \left(4 a + 4\right) x^{13} + 2 x^{7} + \left(4 a + 4\right) x$ \\
\end{tabular}

\begin{tabular}{ll}
$f(x)$  & $f^{-1}(x)$ \\ \hline
$2 x^{19} + 2 x^{13} + 3 x^{7} + 2 x$ & $2 x^{19} + 2 x^{13} + 3 x^{7} + 2 x$ \\
$\left(2 a + 1\right) x^{19} + 2 x^{13} + \left(3 a + 4\right) x^{7} + 2 x$ & $\left(3 a + 3\right) x^{19} + 2 x^{13} + \left(2 a + 2\right) x^{7} + 2 x$ \\
$\left(4 a + 4\right) x^{19} + 2 x^{13} + \left(a + 1\right) x^{7} + 2 x$ & $\left(a + 3\right) x^{19} + 2 x^{13} + \left(4 a + 2\right) x^{7} + 2 x$ \\
$4 x^{19} + 2 x^{13} + x^{7} + 2 x$ & $4 x^{19} + 2 x^{13} + x^{7} + 2 x$ \\
$\left(4 a + 2\right) x^{19} + 2 x^{13} + \left(a + 3\right) x^{7} + 2 x$ & $\left(a + 1\right) x^{19} + 2 x^{13} + \left(4 a + 4\right) x^{7} + 2 x$ \\
$\left(3 a + 3\right) x^{19} + 2 x^{13} + \left(2 a + 2\right) x^{7} + 2 x$ & $\left(2 a + 1\right) x^{19} + 2 x^{13} + \left(3 a + 4\right) x^{7} + 2 x$ \\
$3 x^{19} + 2 x^{13} + 2 x^{7} + 2 x$ & $3 x^{19} + 2 x^{13} + 2 x^{7} + 2 x$ \\
\end{tabular}

\begin{tabular}{ll}
$f(x)$  & $f^{-1}(x)$ \\ \hline
$2 x^{19} + \left(4 a + 4\right) x^{13} + 3 x^{7} + \left(4 a + 4\right) x$ & $2 x^{19} + \left(a + 3\right) x^{13} + 3 x^{7} + \left(a + 3\right) x$ \\
$\left(2 a + 1\right) x^{19} + \left(4 a + 4\right) x^{13} + \left(3 a + 4\right) x^{7} + \left(4 a + 4\right) x$ & $\left(3 a + 3\right) x^{19} + \left(a + 3\right) x^{13} + \left(2 a + 2\right) x^{7} + \left(a + 3\right) x$ \\
$\left(4 a + 4\right) x^{19} + \left(4 a + 4\right) x^{13} + \left(a + 1\right) x^{7} + \left(4 a + 4\right) x$ & $\left(a + 3\right) x^{19} + \left(a + 3\right) x^{13} + \left(4 a + 2\right) x^{7} + \left(a + 3\right) x$ \\
$4 x^{19} + \left(4 a + 4\right) x^{13} + x^{7} + \left(4 a + 4\right) x$ & $4 x^{19} + \left(a + 3\right) x^{13} + x^{7} + \left(a + 3\right) x$ \\
$\left(4 a + 2\right) x^{19} + \left(4 a + 4\right) x^{13} + \left(a + 3\right) x^{7} + \left(4 a + 4\right) x$ & $\left(a + 1\right) x^{19} + \left(a + 3\right) x^{13} + \left(4 a + 4\right) x^{7} + \left(a + 3\right) x$ \\
$\left(3 a + 3\right) x^{19} + \left(4 a + 4\right) x^{13} + \left(2 a + 2\right) x^{7} + \left(4 a + 4\right) x$ & $\left(2 a + 1\right) x^{19} + \left(a + 3\right) x^{13} + \left(3 a + 4\right) x^{7} + \left(a + 3\right) x$ \\
$3 x^{19} + \left(4 a + 4\right) x^{13} + 2 x^{7} + \left(4 a + 4\right) x$ & $3 x^{19} + \left(a + 3\right) x^{13} + 2 x^{7} + \left(a + 3\right) x$ \\
\end{tabular}

\begin{tabular}{ll}
$f(x)$  & $f^{-1}(x)$ \\ \hline
$2 x^{19} + \left(4 a + 2\right) x^{13} + 3 x^{7} + \left(4 a + 2\right) x$ & $2 x^{19} + \left(a + 1\right) x^{13} + 3 x^{7} + \left(a + 1\right) x$ \\
$\left(2 a + 1\right) x^{19} + \left(4 a + 2\right) x^{13} + \left(3 a + 4\right) x^{7} + \left(4 a + 2\right) x$ & $\left(3 a + 3\right) x^{19} + \left(a + 1\right) x^{13} + \left(2 a + 2\right) x^{7} + \left(a + 1\right) x$ \\
$\left(4 a + 4\right) x^{19} + \left(4 a + 2\right) x^{13} + \left(a + 1\right) x^{7} + \left(4 a + 2\right) x$ & $\left(a + 3\right) x^{19} + \left(a + 1\right) x^{13} + \left(4 a + 2\right) x^{7} + \left(a + 1\right) x$ \\
$4 x^{19} + \left(4 a + 2\right) x^{13} + x^{7} + \left(4 a + 2\right) x$ & $4 x^{19} + \left(a + 1\right) x^{13} + x^{7} + \left(a + 1\right) x$ \\
$\left(4 a + 2\right) x^{19} + \left(4 a + 2\right) x^{13} + \left(a + 3\right) x^{7} + \left(4 a + 2\right) x$ & $\left(a + 1\right) x^{19} + \left(a + 1\right) x^{13} + \left(4 a + 4\right) x^{7} + \left(a + 1\right) x$ \\
$\left(3 a + 3\right) x^{19} + \left(4 a + 2\right) x^{13} + \left(2 a + 2\right) x^{7} + \left(4 a + 2\right) x$ & $\left(2 a + 1\right) x^{19} + \left(a + 1\right) x^{13} + \left(3 a + 4\right) x^{7} + \left(a + 1\right) x$ \\
$3 x^{19} + \left(4 a + 2\right) x^{13} + 2 x^{7} + \left(4 a + 2\right) x$ & $3 x^{19} + \left(a + 1\right) x^{13} + 2 x^{7} + \left(a + 1\right) x$ \\
\end{tabular}

}

 \begin{tabular}{llll}
 $r$ & $k$ & \mbox{involutions} &  \mbox{nonzero fixed points}  \\ \hline
$\left(1, 1\right)$ & $\left(0, 0\right)$ & $x$ & $24$ \\
$\left(1, 1\right)$ & $\left(0, 12\right)$ & $x^{13}$ & $12$ \\
$\left(1, 1\right)$ & $\left(1, 23\right)$ & $\left(2 a + 1\right) x^{13} + \left(4 a + 4\right) x$ & $0$ \\
$\left(1, 1\right)$ & $\left(3, 21\right)$ & $\left(a + 2\right) x^{13} + \left(3 a + 1\right) x$ & $0$ \\
$\left(1, 1\right)$ & $\left(5, 19\right)$ & $\left(3 a + 3\right) x^{13} + \left(a + 3\right) x$ & $0$ \\
$\left(1, 1\right)$ & $\left(7, 17\right)$ & $\left(3 a + 3\right) x^{13} + \left(4 a + 2\right) x$ & $0$ \\
$\left(1, 1\right)$ & $\left(9, 15\right)$ & $\left(a + 2\right) x^{13} + \left(2 a + 4\right) x$ & $0$ \\
\end{tabular}

\begin{tabular}{llll}
 $r$ & $k$ & \mbox{involutions} &  \mbox{nonzero fixed points}  \\ \hline
$\left(1, 11\right)$ & $\left(12, 0\right)$ & $2 x^{23} + 2 x^{13} + 3 x^{11} + 2 x$ & $0$ \\
$\left(1, 11\right)$ & $\left(12, 2\right)$ & $\left(2 a + 1\right) x^{23} + 2 x^{13} + \left(3 a + 4\right) x^{11} + 2 x$ & $2$ \\
$\left(1, 11\right)$ & $\left(12, 4\right)$ & $\left(4 a + 4\right) x^{23} + 2 x^{13} + \left(a + 1\right) x^{11} + 2 x$ & $0$ \\
$\left(1, 11\right)$ & $\left(12, 6\right)$ & $4 x^{23} + 2 x^{13} + x^{11} + 2 x$ & $2$ \\
$\left(1, 11\right)$ & $\left(12, 8\right)$ & $\left(4 a + 2\right) x^{23} + 2 x^{13} + \left(a + 3\right) x^{11} + 2 x$ & $0$ \\
$\left(1, 11\right)$ & $\left(12, 10\right)$ & $\left(3 a + 3\right) x^{23} + 2 x^{13} + \left(2 a + 2\right) x^{11} + 2 x$ & $2$ \\
$\left(1, 11\right)$ & $\left(12, 12\right)$ & $3 x^{23} + 2 x^{13} + 2 x^{11} + 2 x$ & $0$ \\
\end{tabular}

\begin{tabular}{llll}
 $r$ & $k$ & \mbox{involutions} &  \mbox{nonzero fixed points}  \\ \hline
$\left(5, 5\right)$ & $\left(12, 0\right)$ & $4 x^{17}$ & $0$ \\
$\left(5, 5\right)$ & $\left(12, 4\right)$ & $\left(4 a + 1\right) x^{17} + \left(a + 3\right) x^{5}$ & $4$ \\
$\left(5, 5\right)$ & $\left(12, 8\right)$ & $\left(4 a + 4\right) x^{17} + a x^{5}$ & $0$ \\
$\left(5, 5\right)$ & $\left(12, 12\right)$ & $4 x^{5}$ & $4$ \\
$\left(5, 5\right)$ & $\left(12, 16\right)$ & $\left(a + 3\right) x^{17} + \left(4 a + 1\right) x^{5}$ & $0$ \\
$\left(5, 5\right)$ & $\left(12, 20\right)$ & $a x^{17} + \left(4 a + 4\right) x^{5}$ & $4$ \\
$\left(5, 5\right)$ & $\left(13, 7\right)$ & $a x^{17} + 3 a x^{5}$ & $0$ \\
\end{tabular}

\begin{tabular}{llll}
 $r$ & $k$ & \mbox{involutions} &  \mbox{nonzero fixed points}  \\ \hline
$\left(5, 11\right)$ & $\left(0, 0\right)$ & $2 x^{23} + 3 x^{17} + 3 x^{11} + 3 x^{5}$ & $4$ \\
$\left(5, 11\right)$ & $\left(0, 2\right)$ & $\left(2 a + 1\right) x^{23} + 3 x^{17} + \left(3 a + 4\right) x^{11} + 3 x^{5}$ & $6$ \\
$\left(5, 11\right)$ & $\left(0, 4\right)$ & $\left(4 a + 4\right) x^{23} + 3 x^{17} + \left(a + 1\right) x^{11} + 3 x^{5}$ & $4$ \\
$\left(5, 11\right)$ & $\left(0, 6\right)$ & $4 x^{23} + 3 x^{17} + x^{11} + 3 x^{5}$ & $6$ \\
$\left(5, 11\right)$ & $\left(0, 8\right)$ & $\left(4 a + 2\right) x^{23} + 3 x^{17} + \left(a + 3\right) x^{11} + 3 x^{5}$ & $4$ \\
$\left(5, 11\right)$ & $\left(0, 10\right)$ & $\left(3 a + 3\right) x^{23} + 3 x^{17} + \left(2 a + 2\right) x^{11} + 3 x^{5}$ & $6$ \\
$\left(5, 11\right)$ & $\left(0, 12\right)$ & $3 x^{23} + 3 x^{17} + 2 x^{11} + 3 x^{5}$ & $4$ \\
\end{tabular}

\begin{tabular}{llll}
 $r$ & $k$ & \mbox{involutions} &  \mbox{nonzero fixed points}  \\ \hline
$\left(5, 11\right)$ & $\left(16, 0\right)$ & $2 x^{23} + \left(4 a + 4\right) x^{17} + 3 x^{11} + \left(4 a + 4\right) x^{5}$ & $4$ \\
$\left(5, 11\right)$ & $\left(16, 2\right)$ & $\left(2 a + 1\right) x^{23} + \left(4 a + 4\right) x^{17} + \left(3 a + 4\right) x^{11} + \left(4 a + 4\right) x^{5}$ & $6$ \\
$\left(5, 11\right)$ & $\left(16, 4\right)$ & $\left(4 a + 4\right) x^{23} + \left(4 a + 4\right) x^{17} + \left(a + 1\right) x^{11} + \left(4 a + 4\right) x^{5}$ & $4$ \\
$\left(5, 11\right)$ & $\left(16, 6\right)$ & $4 x^{23} + \left(4 a + 4\right) x^{17} + x^{11} + \left(4 a + 4\right) x^{5}$ & $6$ \\
$\left(5, 11\right)$ & $\left(16, 8\right)$ & $\left(4 a + 2\right) x^{23} + \left(4 a + 4\right) x^{17} + \left(a + 3\right) x^{11} + \left(4 a + 4\right) x^{5}$ & $4$ \\
$\left(5, 11\right)$ & $\left(16, 10\right)$ & $\left(3 a + 3\right) x^{23} + \left(4 a + 4\right) x^{17} + \left(2 a + 2\right) x^{11} + \left(4 a + 4\right) x^{5}$ & $6$ \\
$\left(5, 11\right)$ & $\left(16, 12\right)$ & $3 x^{23} + \left(4 a + 4\right) x^{17} + 2 x^{11} + \left(4 a + 4\right) x^{5}$ & $4$ \\
\end{tabular}

\begin{tabular}{llll}
 $r$ & $k$ & \mbox{involutions} &  \mbox{nonzero fixed points}  \\ \hline
$\left(7, 5\right)$ & $\left(12, 16\right)$ & $2 x^{19} + \left(a + 1\right) x^{17} + 2 x^{7} + \left(4 a + 4\right) x^{5}$ & $6$ \\
$\left(7, 5\right)$ & $\left(12, 20\right)$ & $2 x^{19} + \left(a + 3\right) x^{17} + 2 x^{7} + \left(4 a + 2\right) x^{5}$ & $10$ \\
$\left(7, 5\right)$ & $\left(18, 0\right)$ & $4 x^{19} + 2 x^{17} + 4 x^{7} + 3 x^{5}$ & $0$ \\
$\left(7, 5\right)$ & $\left(18, 4\right)$ & $4 x^{19} + \left(4 a + 4\right) x^{17} + 4 x^{7} + \left(a + 1\right) x^{5}$ & $4$ \\
$\left(7, 5\right)$ & $\left(18, 8\right)$ & $4 x^{19} + \left(4 a + 2\right) x^{17} + 4 x^{7} + \left(a + 3\right) x^{5}$ & $0$ \\
$\left(7, 5\right)$ & $\left(18, 12\right)$ & $4 x^{19} + 3 x^{17} + 4 x^{7} + 2 x^{5}$ & $4$ \\
$\left(7, 5\right)$ & $\left(18, 16\right)$ & $4 x^{19} + \left(a + 1\right) x^{17} + 4 x^{7} + \left(4 a + 4\right) x^{5}$ & $0$ \\
\end{tabular}

\begin{tabular}{llll}
 $r$ & $k$ & \mbox{involutions} &  \mbox{nonzero fixed points}  \\ \hline
$\left(7, 11\right)$ & $\left(6, 0\right)$ & $2 x^{23} + x^{19} + 3 x^{11} + x^{7}$ & $0$ \\
$\left(7, 11\right)$ & $\left(6, 2\right)$ & $\left(2 a + 1\right) x^{23} + x^{19} + \left(3 a + 4\right) x^{11} + x^{7}$ & $2$ \\
$\left(7, 11\right)$ & $\left(6, 4\right)$ & $\left(4 a + 4\right) x^{23} + x^{19} + \left(a + 1\right) x^{11} + x^{7}$ & $0$ \\
$\left(7, 11\right)$ & $\left(6, 6\right)$ & $4 x^{23} + x^{19} + x^{11} + x^{7}$ & $2$ \\
$\left(7, 11\right)$ & $\left(6, 8\right)$ & $\left(4 a + 2\right) x^{23} + x^{19} + \left(a + 3\right) x^{11} + x^{7}$ & $0$ \\
$\left(7, 11\right)$ & $\left(6, 10\right)$ & $\left(3 a + 3\right) x^{23} + x^{19} + \left(2 a + 2\right) x^{11} + x^{7}$ & $2$ \\
$\left(7, 11\right)$ & $\left(6, 12\right)$ & $3 x^{23} + x^{19} + 2 x^{11} + x^{7}$ & $0$ \\
\end{tabular}

\begin{tabular}{llll}
 $r$ & $k$ & \mbox{involutions} &  \mbox{nonzero fixed points}  \\ \hline
$\left(11, 1\right)$ & $\left(12, 0\right)$ & $2 x^{23} + 2 x^{13} + 2 x^{11} + 3 x$ & $14$ \\
$\left(11, 1\right)$ & $\left(12, 12\right)$ & $2 x^{23} + 3 x^{13} + 2 x^{11} + 2 x$ & $2$ \\
$\left(11, 1\right)$ & $\left(14, 0\right)$ & $\left(2 a + 1\right) x^{23} + 2 x^{13} + \left(2 a + 1\right) x^{11} + 3 x$ & $12$ \\
$\left(11, 1\right)$ & $\left(14, 12\right)$ & $\left(2 a + 1\right) x^{23} + 3 x^{13} + \left(2 a + 1\right) x^{11} + 2 x$ & $0$ \\
$\left(11, 1\right)$ & $\left(16, 0\right)$ & $\left(4 a + 4\right) x^{23} + 2 x^{13} + \left(4 a + 4\right) x^{11} + 3 x$ & $14$ \\
$\left(11, 1\right)$ & $\left(16, 12\right)$ & $\left(4 a + 4\right) x^{23} + 3 x^{13} + \left(4 a + 4\right) x^{11} + 2 x$ & $2$ \\
$\left(11, 1\right)$ & $\left(18, 0\right)$ & $4 x^{23} + 2 x^{13} + 4 x^{11} + 3 x$ & $12$ \\
\end{tabular}

\begin{tabular}{llll}
 $r$ & $k$ & \mbox{involutions} &  \mbox{nonzero fixed points}  \\ \hline
$\left(11, 5\right)$ & $\left(12, 0\right)$ & $2 x^{23} + 2 x^{17} + 2 x^{11} + 3 x^{5}$ & $2$ \\
$\left(11, 5\right)$ & $\left(12, 4\right)$ & $2 x^{23} + \left(4 a + 4\right) x^{17} + 2 x^{11} + \left(a + 1\right) x^{5}$ & $6$ \\
$\left(11, 5\right)$ & $\left(12, 8\right)$ & $2 x^{23} + \left(4 a + 2\right) x^{17} + 2 x^{11} + \left(a + 3\right) x^{5}$ & $2$ \\
$\left(11, 5\right)$ & $\left(12, 12\right)$ & $2 x^{23} + 3 x^{17} + 2 x^{11} + 2 x^{5}$ & $6$ \\
$\left(11, 5\right)$ & $\left(12, 16\right)$ & $2 x^{23} + \left(a + 1\right) x^{17} + 2 x^{11} + \left(4 a + 4\right) x^{5}$ & $2$ \\
$\left(11, 5\right)$ & $\left(12, 20\right)$ & $2 x^{23} + \left(a + 3\right) x^{17} + 2 x^{11} + \left(4 a + 2\right) x^{5}$ & $6$ \\
$\left(11, 5\right)$ & $\left(14, 0\right)$ & $\left(2 a + 1\right) x^{23} + 2 x^{17} + \left(2 a + 1\right) x^{11} + 3 x^{5}$ & $0$ \\
\end{tabular}

\begin{tabular}{llll}
 $r$ & $k$ & \mbox{involutions} &  \mbox{nonzero fixed points}  \\ \hline
$\left(11, 7\right)$ & $\left(6, 0\right)$ & $x^{23} + 2 x^{19} + x^{11} + 3 x^{7}$ & $0$ \\
$\left(11, 7\right)$ & $\left(6, 6\right)$ & $x^{23} + 4 x^{19} + x^{11} + x^{7}$ & $6$ \\
$\left(11, 7\right)$ & $\left(6, 12\right)$ & $x^{23} + 3 x^{19} + x^{11} + 2 x^{7}$ & $0$ \\
$\left(11, 7\right)$ & $\left(6, 18\right)$ & $x^{23} + x^{19} + x^{11} + 4 x^{7}$ & $6$ \\
$\left(11, 7\right)$ & $\left(8, 0\right)$ & $\left(a + 3\right) x^{23} + 2 x^{19} + \left(a + 3\right) x^{11} + 3 x^{7}$ & $2$ \\
$\left(11, 7\right)$ & $\left(8, 6\right)$ & $\left(a + 3\right) x^{23} + 4 x^{19} + \left(a + 3\right) x^{11} + x^{7}$ & $8$ \\
$\left(11, 7\right)$ & $\left(8, 12\right)$ & $\left(a + 3\right) x^{23} + 3 x^{19} + \left(a + 3\right) x^{11} + 2 x^{7}$ & $2$ \\
\end{tabular}

\begin{tabular}{llll}
 $r$ & $k$ & \mbox{involutions} &  \mbox{nonzero fixed points}  \\ \hline
$\left(11, 11\right)$ & $\left(2, 0\right)$ & $\left(3 a + 1\right) x^{23} + \left(3 a + 2\right) x^{11}$ & $0$ \\
$\left(11, 11\right)$ & $\left(2, 2\right)$ & $\left(a + 3\right) x^{11}$ & $2$ \\
$\left(11, 11\right)$ & $\left(2, 4\right)$ & $\left(2 a + 3\right) x^{23} + 4 a x^{11}$ & $0$ \\
$\left(11, 11\right)$ & $\left(2, 6\right)$ & $\left(3 a + 3\right) x^{23} + 3 a x^{11}$ & $2$ \\
$\left(11, 11\right)$ & $\left(2, 8\right)$ & $\left(2 a + 1\right) x^{23} + \left(4 a + 2\right) x^{11}$ & $0$ \\
$\left(11, 11\right)$ & $\left(2, 10\right)$ & $\left(a + 2\right) x^{23} + x^{11}$ & $2$ \\
$\left(11, 11\right)$ & $\left(2, 12\right)$ & $\left(3 a + 2\right) x^{23} + \left(3 a + 1\right) x^{11}$ & $0$ \\
\end{tabular}

\begin{tabular}{llll}
 $r$ & $k$ & \mbox{involutions} &  \mbox{nonzero fixed points}  \\ \hline
$\left(11, 11\right)$ & $\left(8, 18\right)$ & $\left(a + 4\right) x^{23} + \left(a + 2\right) x^{11}$ & $4$ \\
$\left(11, 11\right)$ & $\left(8, 20\right)$ & $\left(2 a + 1\right) x^{23}$ & $2$ \\
$\left(11, 11\right)$ & $\left(8, 22\right)$ & $3 a x^{23} + \left(4 a + 1\right) x^{11}$ & $4$ \\
$\left(11, 11\right)$ & $\left(9, 21\right)$ & $\left(3 a + 1\right) x^{23}$ & $0$ \\
$\left(11, 11\right)$ & $\left(10, 0\right)$ & $\left(2 a + 4\right) x^{23} + 2 a x^{11}$ & $0$ \\
$\left(11, 11\right)$ & $\left(10, 2\right)$ & $\left(4 a + 3\right) x^{23} + x^{11}$ & $2$ \\
$\left(11, 11\right)$ & $\left(10, 4\right)$ & $\left(a + 1\right) x^{23} + \left(3 a + 3\right) x^{11}$ & $0$ \\
\end{tabular}

\begin{tabular}{llll}
 $r$ & $k$ & \mbox{involutions} &  \mbox{nonzero fixed points}  \\ \hline
$\left(11, 11\right)$ & $\left(16, 12\right)$ & $\left(4 a + 2\right) x^{23} + \left(4 a + 1\right) x^{11}$ & $2$ \\
$\left(11, 11\right)$ & $\left(16, 14\right)$ & $\left(2 a + 3\right) x^{23} + a x^{11}$ & $4$ \\
$\left(11, 11\right)$ & $\left(16, 16\right)$ & $\left(3 a + 3\right) x^{11}$ & $2$ \\
$\left(11, 11\right)$ & $\left(16, 18\right)$ & $4 a x^{23} + \left(4 a + 3\right) x^{11}$ & $4$ \\
$\left(11, 11\right)$ & $\left(16, 20\right)$ & $2 x^{23} + \left(3 a + 1\right) x^{11}$ & $2$ \\
$\left(11, 11\right)$ & $\left(16, 22\right)$ & $\left(a + 1\right) x^{23} + \left(2 a + 2\right) x^{11}$ & $4$ \\
$\left(11, 11\right)$ & $\left(17, 5\right)$ & $\left(a + 4\right) x^{23}$ & $0$ \\
\end{tabular}



\end{document}